\documentclass[sn-mathphys,Numbered]{sn-jnl}

\usepackage{graphicx}%
\usepackage{multirow}%
\usepackage{amsmath,amssymb,amsfonts}%
\usepackage{amsthm}%
\usepackage{mathrsfs}%
\usepackage[title]{appendix}%
\usepackage{xcolor}%
\usepackage{textcomp}%
\usepackage{manyfoot}%
\usepackage{booktabs}%
\usepackage{algorithm}%
\usepackage{algorithmicx}%
\usepackage{algpseudocode}%
\usepackage{tikz}%
\usepackage{listings}%
\usepackage[pagewise,displaymath]{lineno}

\theoremstyle{thmstyleone}%
\newtheorem{thm}{Theorem}[section]
\newtheorem{prop}[thm]{Proposition}%
\newtheorem{lem}[thm]{Lemma}
\newtheorem{coro}[thm]{Corollary}

\theoremstyle{thmstyletwo}%
\newtheorem{ex}{Example}%
\newtheorem{remark}{Remark}%

\theoremstyle{thmstylethree}%
\newtheorem{defn}{Definition}[section]%

\begin{document}

\title[Zeta functions of graphs of groups]{Zeta functions of geometrically finite graphs of groups}


\author[1]{\fnm{Soonki} \sur{Hong}}\email{soonkihong@postech.ac.kr}

\author*[2]{\fnm{Sanghoon} \sur{Kwon}}\email{skwon@cku.ac.kr}


\affil[1]{\orgdiv{Department of Mathematics}, \orgname{Postech}, \orgaddress{\street{Cheongam-ro}, \city{Pohang}, \postcode{37673}, \state{Kyungbuk}, \country{Korea}}}

\affil*[2]{\orgdiv{Department of Mathematical Education}, \orgname{Catholic Kwandong University}, \orgaddress{\street{Beomil-ro}, \city{Gangneung}, \postcode{25601}, \state{Gangwon}, \country{Korea}}}



\abstract{In this paper, we explore the properties of zeta functions associated with infinite graphs of groups that arise as quotients of cuspidal tree-lattices, including all non-uniform arithmetic quotients of the tree of rank one Lie groups over local fields. Through various examples, we illustrate pairs of non-isomorphic cuspidal tree-lattices with the same Ihara zeta function. Additionally, we analyze the spectral behavior of a sequence of graphs of groups whose pole-free regions of zeta functions converge towards 0, which also presents an example of arbitrary small exponential error-term in counting geodesic formula.}




\maketitle

\section{Introduction}\label{sec:1}

Given a cocompact subgroup $\Gamma$ of $PSL(2,\mathbb{R})$, Selberg introduced a zeta function $Z_\Gamma(s)$ and proved many important properties which resemble those of usual $L$-functions in number theory, such as Euler product, functional equation, and the analogue of Riemann hypothesis (\cite{Se}). This function, now called with Selberg's zeta function, is generalized to any discrete subgroup $\Gamma$ of semi-simple Lie group of $\mathbb{R}$-rank one, even for non-uniform lattice $\Gamma$ (\cite{Gu}) or geometrically finite subgroups of $PSL(2,\mathbb{R})$ with infinite area (\cite{Bo}).

Meanwhile, inspired by Selberg’s zeta function, Ihara \cite{Ih} initially introduced a zeta function that counts prime elements in discrete subgroups of rank one $p$-adic groups. It can be viewed as a geometric zeta function associated with the corresponding finite graph, which is a quotient of the Bruhat-Tits tree for the $p$-adic groups (\cite{Ser}). Over time, it has been generalized by Sunada \cite{Su}, Hashimoto \cite{Ha}, and Bass \cite{Ba2}, revealing connections to number theory and representation theory. This zeta function is defined as the product
\[Z_X(u)=\prod_p\frac{1}{1-u^{\ell(p)}}\]
where $p$ runs through the set of prime cycles in a finite graph $X$. This product, being infinite in general, converges to a rational function and satisfies the \emph{Ihara determinant formula}
\[Z_X(u)=\frac{(1-u^2)^{\chi(X)}}{\det(1-uA+u^2Q)}\]
where $A$ is the adjacency operator, $Q+1$ is the valency operator, and $\chi(X)$ is the Euler characteristic of the graph $X$. One of the most remarkable features of this formula is that if $X=\Gamma\backslash\mathcal{T}$ for some Bruhat-Tits tree $\mathcal{T}$ of a $p$-adic group $G$ and a cocompact arithmetic subgroup $\Gamma$ of $G$, then the right hand side of the formula equals the non-trivial part of the Hasse-Weil zeta function of the Shimura curve attached to $\Gamma$.

In recent years, there has been interest in generalizing these zeta functions to infinite graphs. For example, \cite{CMS} and \cite{Cl} utilizes a finite trace argument on a group von Neumann algebra and define the zeta function as a determinant. Another approach, described in \cite{GZ}, involves approximating an infinite graph by finite ones and defining the zeta function as a suitable limit. In \cite{CJK}, the Ihara zeta function is extended to infinite graphs by considering cycles that pass through a specific point using Heat kernel, rather than considering all cycles. Furthermore, in \cite{LPS}, the zeta function is extended to the case of an infinite graph equipped with an invariant measure and a groupoid action, which is referred to as a measure graph. Meanwhile, \cite{DK} take another approach by considering cycles obtained from geodesics in the universal covering tree of infinite weigted graphs. Especially, they defined the Bass-Ihara zeta function for a weighted graph $(X,w)$ by
\[Z_{(X,w)}(u)=\prod_{[c]}\frac{1}{1-w(c)u^{\ell(c)}}.\]
Here, $c$ runs through the set of prime cycles in a graph $X$. The precise definition of $w(c)$ is given in Section~\ref{sec:3}. It is also proved in \cite{DK} that it converges to a rational function when $\Gamma$ is a cuspidal tree lattice.


Our goal in this paper is to extend the idea of \cite{DK} to develop a comprehensive definition of the zeta function for a graph of groups $(X,\mathcal{G})$. Graphs of groups naturally arise in the study of group actions on trees and their covering theory. The key distinction between the zeta function of weighted graphs and graphs of groups lies in the additional information regarding the centralizers of hyperbolic elements. 

Of particular interest is the case where the graph of groups $(X,\mathcal{G})$ is obtained from a geometrically finite (or equivalently, cuspidal) tree lattice. This encompasses every quotient of the tree of rank-one Lie groups over non-archimedean local fields by its discrete subgroup (as proved in \cite{Lu}). Specifically, we define the zeta function $Z_{(X,\mathcal{G})}(u)$ of $(q+1)$-regular graphs of groups $(X,\mathcal{G})$ that resembles the Selberg zeta function, and investigate its relationship with the Bass-Ihara zeta function $Z_{(X,w)}(u)$ for the weighted graph $(X,w)$ associated with $(X,\mathcal{G})$ under \emph{centrally rigid} condition. (We refer to Definition~\ref{def:cr} for precise details.) It is worth mentioning that the Selberg zeta function and its determinant expression for principal congruence subgroups of $PGL(2,\mathbb{F}_q[t])$ have been also achieved in \cite{Na} with algebraic approach. We build upon these algebraic achievements and provide a geometric perspective by examining zeta functions within the framework of graph of groups.

We define a Selberg zeta function of a graph of groups $(X,\mathcal{G})$ by
\[Z_{(X,\mathcal{G})}(u)=\prod_{[\gamma]\in \textrm{P}(\pi_1(X,\mathcal{G}))} \frac{1}{1-u^{\ell(\gamma)}}\]
where $[\gamma]$ runs over conjugacy classes of primitive hyperbolic elements $\{\gamma\}$ of $\pi_1(X,\mathcal{G})$. Precise definition for a graph of groups $(X,\mathcal{G})$ and its fundamental group $\pi_1(X,\mathcal{G})$ are given in Section~\ref{sec:2}.
If $\Gamma=\pi_1(X,\mathcal{G})$ is a subgroup of $PGL(2,\mathbb{F}_q[t])$, then the Selberg zeta function (of Ruelle type) for $\Gamma<PGL(2,\mathbb{F}_q[t])$ is given by
\[\zeta_\Gamma{(s)}=\prod_{[\gamma]\in\textrm{P}(\Gamma)}\frac{1}{1-N(\gamma)^{-s}}\] where $P(\Gamma)$ is the set of all the primitive hyperbolic conjugacy classes of $\Gamma$ and $N(\gamma)=q^{2\deg(\textrm{tr}\gamma)}$.
Then, it holds that
\[\zeta_\Gamma(s)=Z_{(X,\mathcal{G})}(q^{-s}).\]

Meanwhile, if $(X,w)$ is the weighted graph associated to $(X,\mathcal{G})$, then following the idea of \cite{DK}, we have the Bass-Ihara zeta function of $(X,w)$ given by
\[Z_{(X,w)}(u)=\prod_{[c]}\frac{1}{1-w(c)u^{\ell(c)}}\]
where $w(c)$ is the natural weight counting the number of inequivalent pre-images of $c$ in the universal covering tree $\mathcal{T}$ of $(X,w)$.
The precise definition of $w(c)$ is given in Section~\ref{sec:3}. 

Let $\textrm{Aut}(\mathcal{T}_{q+1})$ be a group of automorphisms of $(q+1)$-regular tree. Let us call a subgroup $\Gamma$ of $\textrm{Aut}(\mathcal{T}_{q+1})$ \emph{centrally rigid} if the cardinality $|\textrm{Cent}(\gamma,\Gamma_{\textrm{Axis}(\gamma)})|$ of the centralizer of $\gamma$ in $\Gamma_{\textrm{Axis}(\gamma)})$ is constant on every $\gamma\in P(\Gamma)$. When $\Gamma$ is centrally rigid, we denote by $c_\Gamma=|\textrm{Cent}(\gamma,\Gamma_{\textrm{Axis}(\gamma)})|$. For example, $c_{PGL(2,\mathbb{F}_q[t])}=q-1$.
The following theorem describes the relation between two zeta functions for centrally rigid subgroups.

\bigskip
\begin{thm}
Let $\mathbf{X}=(X,\mathcal{G})$ be a $(q+1)$-regular graph of groups arising as a quotient by a centrally rigid subgroup $\Gamma$ of $PGL(2,\mathbb{F}_q[t])$ and $\mathbf{A}=(X,w)$ be the associated weighted graph. Then, we have
\[Z_\mathbf{X}(u)=Z_\mathbf{A}(u)^{c_\Gamma}.\]
\end{thm}

\bigskip





In Section~\ref{sec:5}, we compute the Bass-Ihara zeta function for some examples. In particular, we prove the following proposition.

\bigskip

\begin{prop}\label{prop:non-isom}
For every $q\ge 3$, there exist pairs of $(q+1)$-regular graphs of groups whose induced edge-indexed graphs are non-isomorphic but the zeta functions are identical.
\end{prop}

\bigskip

Additionally, we explore a family of graphs of groups introduced in \cite{Ef} which discussed the spectral behavior of automorphic functions under deformations of graphs of groups. We note that both the poles of zeta functions and the spectrum of the combinatorial Laplacian on graphs of groups are related to the distribution of primes. If $R_\mathbf{X}$ is the radius of convergence of $Z_\mathbf{X}(u)$, then there is $\epsilon>0$ such that any pole of $Z_\mathbf{X}(u)$ in the region $R_\mathbf{X}\le |u|\le R_\mathbf{X}+\epsilon$ must lie on the circle $|u|=R_\mathbf{X}$. We call the region $|u|<R_\mathbf{X}+\epsilon$ by \emph{pole-free region} of $\mathbf{X}$. Especially, the pole-free region of the zeta function $Z_\mathbf{X}(u)$ is connected with the error-term of the prime geodesics theorem in graphs of groups.

By investigating the poles of zeta functions, we establish the following theorem. The proof will be given in Section~\ref{sec:6}.

\bigskip

\begin{thm}\label{thm:pole-free}
There is a sequence of $(q+1)$-graphs of groups $\mathbf{X}_N=(X_N,\mathcal{G}_N)$ such that the pole-free regions of $Z_{\mathbf{X}_N}$ converges to $|u|\le \frac{1}{q}$.
\end{thm}

\bigskip
We obtain from this Theorem that such a sequence of graphs of groups gives arbitrary small exponential error term when we count the number of geodesics. Given a graph of groups $\mathbf{X}=(X,\mathcal{G})$ and its fundamental group $\Gamma=\pi_1(X,\mathcal{G})$, let \[
    N_m(\mathbf{X})=\sum_{\ell(\gamma)=m}\frac{\ell(\gamma_0)}{|\textrm{Cent}(\gamma,\Gamma_{\textrm{Axis}(\gamma)})|}.
\] Here, the sum runs over hyperbolic elements $\gamma$ in $\Gamma$ and $\gamma_0$ is the primitive hyperbolic element underlying $\gamma$. Taking logarithmic derivative of the zeta function yields the following.
\bigskip

\begin{coro}
Given any $\epsilon>0$, we have a graph of groups $\mathbf{X}=(X,\mathcal{G})$ such that 
\[\lim_{m\to\infty}\frac{N_m(\mathbf{X})}{q^m}=1\quad\textrm{ but }\quad |N_m(\mathbf{X})-q^m|>(q-\epsilon)^m\]
for arbitrary large $m$.
\end{coro} 

\bigskip

This paper is organized as follows. We prepare basic setup about graphs of groups and cuspidal tree lattices in Section~\ref{sec:2}. In Section~\ref{sec:3}, we define zeta function of graphs of groups and compare it with the Bass-Ihara zeta function of associated weighted graphs. More examples are provided in Section~\ref{sec:5} and \ref{sec:6}.

\subsubsection*{Acknowledgement} The first author is supported by Basic Science Research Institute Fund, whose NRF grant number is 2021R1A6A1A10042944. The second author is supported by NRF of Korea grant RS-2023-0023781122682121230001 and Samsung STF Project no. SSTF-BA2101-01. Both authors thank POSTECH Mathematics Institute, where this work was completed, for its hospitality.

\bigskip


\section{Graphs of groups and tree lattices}\label{sec:2}




In this section, we review the basic setup about graphs of groups and associated weighted graphs. The Bass-Serre theory, being developed by Serre \cite{Se} and the substantial contribution of the subsequent work of Bass \cite{Ba}, deals with the theory of group actions on graphs and their covering theory via the notion of graphs of groups.

\subsection{Graphs of groups}

A graph $X$ consists of a set of vertices $Ver(X)$ and a set of oriented edges $Ed(X)$ together with maps
\begin{align*}
Ed(X)\to Ver(X)\times Ver(X), \qquad e\mapsto (s(e),t(e))\,\,(\textrm{source and target})
\end{align*}
and
\begin{align*}
Ed(X)\to Ed(X),\qquad e\mapsto \overline{e}\,\, (\textrm{inverse})
\end{align*}
such that $\overline{\overline{e}}=e$, $\overline{e}\ne e$ and $s(e)=t(\overline{e})$. We will sometimes write $X$ for $Ver(X)$.

\bigskip

\begin{defn}
A graph of groups $\mathbf{X}$ is a pair $(X,\mathcal{G})$ where $X$ is a graph and $\mathcal{G}$ is a collection of groups $G_x$ and $G_e$ attached to each vertex and edge of $X$, together with the condition $G_e=G_{\overline{e}}$ and monomorphisms $\phi_e\colon G_e\to G_{t(e)}$.
\end{defn}

\bigskip

Upon choosing a maximal subtree $T$ of $X$, we can define the fundamental group $\pi_1=\pi_1(\mathbf{X},T)$ as follows. The group $\pi_1$ is generated by $G_x$ $(x\in Ver(X))$ and by elements $g_e$ $(e\in Ed(X))$, subject to the relations
\[g_e\phi_e(g)g_e^{-1}=\phi_{\overline{e}}(g),\qquad g_{\overline{e}}=g_e^{-1}\]
and $g_e=1$ for $e\in Ed(T)$.

Given a group action on a graph, the quotient graph has a natural structure of graph of groups. Let $\Gamma$ be a group acting on $Y$ without inversions of edges (that is, $ge\ne\overline{e}$ for all $g\in\Gamma$) and $X$ be the quotient graph $X=\Gamma\backslash Y$. Choose a maximal subtree $R$ of $X$ and lift it to a subtree $T$ of $Y$. Let $E$ be the representatives of the edges of $X-R$ and lift $E$ to $F(\subset Ed(Y))$ whose sources are in $T$. Let $y_x$ be the vertex of $Y$ lifted from $x\in Ver(X)$ and let $\Gamma_{y_x}$ be the corresponding stabilizer. Similarly, let $f_e$ be the edge of $Y$ lifted from $e\in Ed(X)$ and let $\Gamma_{f_e}$ be the corresponding stabilizer. 

For each edge $f_e\in F$ lifted from $e\in E$, there is a $\gamma_e\in \Gamma$ such that $\gamma_et(f_e)\in Ver(T)$. Since 
\[\Gamma_{\gamma_et(f_e)}=\gamma_e\Gamma_{t(f_e)}\gamma_e^{-1},\] it follows that $\Gamma_{f_e}\subset \Gamma_{t(f_e)}=\gamma_e^{-1}\Gamma_{\gamma_et(f_e)}\gamma_e$.

\bigskip

\begin{thm}[\cite{Ser} Theorem 1.13]
The group $\Gamma$ is generated by $\{\Gamma_{y_x}\colon x\in Ver(X)\}$ and $\{\gamma_e\colon e\in E\}$, subject to the relations 
\[\gamma=\gamma_e^{-1}\gamma'\gamma_e\] for $\gamma\in \Gamma_{y_x}$ with $\gamma'\in \Gamma_{\gamma_et(f_e)}$ as above.
\end{thm}

\bigskip

Thus, the pair $(X,\mathcal{G})$ where $\mathcal{G}$ is the collection of these groups $\Gamma_x$ and $\Gamma_e$ is a graph of groups. We say this is a \emph{quotient graph of groups} for $\Gamma\curvearrowright Y$.

\subsection{Cuspidal tree lattices}

A \emph{tree lattice} is a group $\Gamma$ together with an action on a locally finite tree $\mathcal{T}$ such that every stabilizer group $\Gamma_x$ of each vertex $x\in Ver(\mathcal{T})$ is finite and
\[\textrm{covol}(\Gamma\curvearrowright\mathcal{T})\overset{\textrm{def}}{=}\sum_{x\in X}\frac{1}{|\Gamma_x|}<\infty\] for the quotient graph $X=\Gamma\backslash\mathcal{T}$.

A (finite or infinite) sequence $P=(e_1,\ldots, e_n,\ldots)$ of edges in a graph $X$ is called \emph{path} if $t(e_i)=s(e_{i+1})$. It is said to be without backtracking if $s(e_i)\ne t(e_{i+1})$ for all $i\ge 1$. A ray $r=\{e_1,e_2,\ldots\}$ in $\mathcal{T}$ is an infinite path without backtracking such that $e_i\ne e_j,\overline{e_j}$ for all $i\le j$. Two rays $r$ and $s$ are said to be \emph{equivalent} if they join at some point, that is, if there exist $N\in \mathbb{Z}_{\ge 0}$ and $k\in\mathbb{Z}_{\ge 0}$ such that $r_{j+k}=s_j$ for all $j\ge N$. An equivalence class of rays is called an \emph{end} of $\mathcal{T}$. The set $\partial_\infty\mathcal{T}$ of all ends is called the \emph{visual boundary} of $\mathcal{T}$. For a given vertex $x_0$ and any point $c\in\partial_\infty \mathcal{T}$, there exists a unique ray in $c$ that starts at the point $x_0$. 

\bigskip

\begin{defn}[\cite{Pa2}, Theorem 1.1]
A tree lattice $\Gamma$ acting on $\mathcal{T}$ is called \emph{geometrically finite} (or \emph{cuspidal}) if the quotient graph of groups $(X,\mathcal{G})$ has the following form.
\begin{enumerate}
    \item There is a finite subset $C$ of $X=\Gamma\backslash \mathcal{T}$ such that $X-C$ is a union of finitely many rays $r_1,r_2,\ldots,r_c$.
    \item Each ray $r=\{e_1,e_2,\ldots\}$ in $X-C$ satisfies the following: If $s(e_1)$ is adjacent to some vertex of $C$ and $e_i$'s go towards at ends, then $\Gamma_{s(e_i)}=\Gamma_{e_i}\subset \Gamma_{s(e_{i+1})}$ for every $i\ge 1$ and all these groups $\Gamma_{s(e_i)}$ are finite subgroups of $\Gamma_\omega$ where $\omega$ denotes the end of $\mathcal{T}$ represented by a lifting of $r$.
\end{enumerate}
\end{defn}

\bigskip

We recall that \cite{Lu} gives the general structure theorem for lattices in $F$-rank one semi-simple group $G$ over a non-Archimedean local field $F$.

\bigskip

\begin{thm}[\cite{Lu}, Theorem 6.1]
Let $F$ be a locally compact non-archimedean field, $G$ a groups of $F$-points of a semi-simple connected algebraic $F$-group of $F$-rank one, $\mathcal{T}$ the Bruhat-Tits tree associated with $G$, and $\Gamma$ a lattice in $G$. Then as a subgroup of $\textrm{Aut}(\mathcal{T})$, the group $\Gamma$ is cuspidal.
\end{thm}

\bigskip


\section{Bass-Ihara zeta function and Selberg zeta function of tree lattices}\label{sec:3}
In this section, we define Selberg zeta function of graphs of groups and compare with Bass-Ihara zeta function of weighted graphs. We also present a determinant formula for the zeta function in terms of the edge-adjacency operator.

\subsection{Zeta function of tree lattices}

Given a path $P=(e_1,\ldots, e_n)$ in a graph $X$, we call $n$ the length of $P$ and denote it by $\ell(P)$. A path $P=(e_1,\ldots,e_n)$ is \emph{closed} if $t(e_n)=s(e_1)$. We say two closed paths $(e_1,\ldots,e_n)$ and $(f_1,\ldots, f_n)$ are \emph{equivalent} if there is $k\in\mathbb{Z}/n\mathbb{Z}$ such that $e_i=f_{i+k}$ for all $i\in\mathbb{Z}/n\mathbb{Z}$. An equivalence class of closed path is called \emph{cycle}. If a cycle $C$ satisfies $C \ne D^{f}$ for any $f > 1$ and a closed path $D$ in $G$, then $C$ is called a \emph{primitive cycle}. Hence, for any cycle $C$, there is a unique primitive cycle $C_0$ such that $C=C_0^m$ for some $m\ge 1$. A cycle $C=[(e_1,\ldots,e_n)]$ is called \emph{reduced} if it consists only of closed paths without backtracking, that is, $\overline{e_i}\ne e_{i+1}$ for all $1\le i\le n$.

We begin by reviewing the zeta function of finite graphs. For a finite graph $X$, the Ihara zeta function is defined at a complex number $u$, for which $\left| u \right|$  is sufficiently small, by 
\[Z_{X}(u)=\prod_{[C]}\frac{1}{1-u^{\ell(C)}}\]
  where the product is taken over all reduced primitive cycles $\left[C\right]$ in $X$. The product, being infinite in general, converges to a reciprocal of a polynomial. In fact, it satisfies Ihara determinant formula
\[Z_X(u)=\frac{1}{(1-u^2)^{-\chi}\det(1-uA+u^2Q)},\]
where $A$ is the adjacency operator of the graph, $Q+1$ is the valency operator, and $\chi$ is the Euler characteristic $|Ver(X)|-\frac{|Ed(X)|}{2}$ of the graph $X$.

These zeta functions are generalized to certain infinite graphs. Indeed, the authors in \cite{DK} extended the theory of Ihara zeta functions to cuspidal lattices, and derived a prime geodesic theorem for arithmetic lattices of rank one Lie groups over local fields. 

Suppose now that $G$ is a group acting on an infinite regular tree $\mathcal{T}$. We may classify elements in $G$ that act without inversions on $\mathcal{T}$, as follows:
\begin{enumerate}
    \item[(a)] identity
    \item[(b)] elliptic: elements which have fixed vertices on $\mathcal{T}$
    \item[(c)] hyperbolic: elements which have no fixed vertices on $\mathcal{T}$ (and hence have two fixed points on $\partial_\infty \mathcal{T}$)
\end{enumerate}

Given a hyperbolic element $g\in G$, let 
\[\ell(g)=\min\{d(x,gx)\colon x\in Ver(\mathcal{T})\}.\] Then the subset $\textrm{Axis}(g)$ of vertices in $\mathcal{T}$ given by $\{x\in Ver(\mathcal{T})\colon d(x,gx)=\ell(g)\}$ constitutes an infinite path in $\mathcal{T}$ and $g$ induces a shift by the length $\ell(g)$ on $\textrm{Axis}(g)$. Here, $\ell(g)$ is called the \emph{translation length} of $g$ (\cite{Ser} p.63, see also \cite{Na}). A hyperbolic element $g$ is called \emph{primitive} if $g\ne h^r$ for any $h\in G$ and $r>1$.

Furthermore, in the case where $\Gamma=PGL(2,\mathbb{F}_q[t])$, the following is more detailed classification from the view of actions on the Bruhat-Tits tree $\mathcal{T}$ and its geometric boundary $\partial_\infty\mathcal{T}$ at infinity.

\begin{enumerate}
    \item[(a)] identity
    \item[(b)] elliptic: elements which have fixed vertices on $\mathcal{T}$ and no fixed points on $\partial_\infty \mathcal{T}$.
    \item[(c)] parabolic: elements which have fixed vertices on $\mathcal{T}$ and a fixed point on $\partial_\infty \mathcal{T}$.
    \item[(d)] split hyperbolic: elements which have fixed vertices on $\mathcal{T}$ and two fixed points on $\partial_\infty \mathcal{T}$.
    \item[(e)] hyperbolic: elements which have no fixed vertices on $\mathcal{T}$ (and hence have two fixed points on $\partial_\infty \mathcal{T}$)
\end{enumerate}

It is particularly useful in the classification of conjugacy classes in $PGL(2,\mathbb{F}_q[t])$, which is analogous to the case of $PSL(2,\mathbb{Z})$. Let $C^\times$ be a complete set of representatives of equivalence classes in $\mathbb{F}_q^\times\backslash\{1\}$ defined by the relation $ab=1$ anc $C_+$ for a complete set representatives of equivalence classes in $\mathbb{F}_q$ defined by the relation $a+b=0$. Let $I$ be the set of monic polynomials in $\mathbb{F}_q[t]$ and $D$ be the subset of $I$ consisting of square-free monic polynomials of even degree. Suppose that $\omega=x+y\sqrt{d}$ is a quadratic irrational functions with $x,y\in\mathbb{F}_q(t)$, $\omega$ satisfies the quadratic equation
\[C\omega^2-B\omega+A=0,\quad A,B,C\in\mathbb{F}_q[t]\]
with $\textrm{gcd}(A,B,C)=1$, and the coefficient of the highest power in $t$ of $2Cy$ is $1$. Then, the polynomicals $A,B,C$ are uniquely determined. In this case, we denote by $\{A,B,C\}=\omega$. Under this setting, the discriminant of $\omega$, defined by $B^2-4AC$, is eqaul to $4C^2y^2d$. 

We recall the conjugacy classes classification result of the group $PGL(2,\mathbb{F}_q[t])$.

\bigskip

\begin{prop}[\cite{Na}, Proposition~2.1]\label{prop:pgl2conj} Let $q$ be an odd prime power and $\alpha$ be a generator of $\mathbb{F}_q^\times$. Let $h_{l\sqrt{d}}$ be the narrow class number of the order $\mathcal{O}_{l\sqrt{d}}$ in $\mathbb{F}_q(t)(\sqrt{d})$ and $\epsilon_{l\sqrt{d}}=t_0+u_0l\sqrt{d}$ be an fundamental unit of $\mathcal{O}_{l\sqrt{d}}$. For every real quadratic irrational function $\omega=\{A,B,C\}$ of discriminant $dl^2$, let
\[\gamma_\omega=\begin{pmatrix} (t_0+Bu_0)/2 & -Au_0 \\ Cu_0 & (t_0-Bu_0)/2\end{pmatrix}.\]
Then, a complete set of representatives of conjugacy classes of $PGL(2,\mathbb{F}_q[t])$ is given by the following five types of elements:
\begin{enumerate}
    \item identity
    \item elliptic: $\left\{\begin{pmatrix} a/2 & \alpha/4 \\ 1 & a/2 \end{pmatrix}\right\}$, $(a\in C_+)$
    \item parabolic: $\left\{\begin{pmatrix} 1 & x \\ 0 & 1 \end{pmatrix}\right\}$, $(x\in I)$
    \item split hyperbolic: $\left\{\begin{pmatrix}1 & 0 \\ 0 & c \end{pmatrix}\right\}$, $(c\in C^\times)$
    \item hyperbolic: $\{\gamma_\omega^n\}$ where $d\in D, l\in I, n=1,2,3,\ldots$ and $\omega$ runs through a complete set of representative of $\Gamma$-equivalence classes of the real quadratic irrational functions of discriminant $dl^2$.
\end{enumerate}
\end{prop}

\bigskip

Now we define a zeta function of a graph of groups.

\bigskip

\begin{defn}
Given a graph of groups $(X,\mathcal{G})$, we define the zeta function of $(X,\mathcal{G})$ by 
\[Z_{(X,\mathcal{G})}(u)=\prod_{[\gamma]}\frac{1}{1-u^{\ell(\gamma)}}\]
where the product runs over conjugacy classes $[\gamma]$ of primitive hyperbolic elements $\{\gamma\}$ of $\pi_1=\pi_1(X,\mathcal{G})$.
\end{defn}

\bigskip

To a graph of groups $(X,\mathcal{G})$, we associate the weighted graph $(X,w)$ as follows. Given an edge $e$ of a graph of finite groups $(X,\mathcal{G})$, let $w(e)$ be the ratio $\frac{|G_{s(e)}|}{|G_e|}$ of the size of attached groups $G_{s(e)}$ and $G_e$.
Now given a closed path $P=(e_1,\ldots,e_n)$ in $(X,\mathcal{G})$, let
\[w(p)=\prod_{j\in\mathbb{Z}/n\mathbb{Z}}w(e_j,e_{j+1})\]
where $$w(e,e')=\left\{\begin{array}{ll}w(e'), & e'\ne \overline{e}\\
w(e')-1, & e'=\overline{e}\end{array}\right..$$
Then, if we write
\[\frac{uZ'_{(X,\mathcal{G})}(u)}{Z_{(X,\mathcal{G})}(u)}=\sum_{m=1}^{\infty} R_mu^m\] by taking logarithmic derivative of $Z_{(X,\mathcal{G})}(u)$, then this yields
\begin{equation}\label{eqn:Rm}
Z_{(X,\mathcal{G})}(u)=\exp\left(\sum_{m=1}^{\infty}R_m\frac{u^m}{m}\right)
\end{equation}
and
\[R_m=\sum_{[\gamma]\colon \ell(\gamma)=m}\ell(\gamma_0)\] where $\gamma_0$ denotes the primitive hyperbolic element underlying $\gamma$.

Meanwhile, according to the proof of the Proposition~8.4 of \cite{DK}, the zeta function for a weighted graph $(X,w)$ satisfies
\[Z_{(X,w)}(u)=\prod_{[C]}\frac{1}{1-w(C)u^{\ell(C)}}=\exp\left(\sum_{m=1}^{\infty}N_m\frac{u^m}{m}\right)\]
where
\begin{equation}\label{eqn:Nm}
    N_m=\sum_{ \ell(c)=m}w(c)\ell(c_0)=\sum_{\ell(\gamma)=m}\frac{\ell(\gamma_0)}{|\textrm{Cent}(\gamma,\Gamma_{\textrm{Axis}(\gamma)})|}.
\end{equation}

\bigskip

\begin{defn}\label{def:cr}

A discrete group $\Gamma$ acting on a $(q+1)$-regular tree $\mathcal{T}_{q+1}$ is called \emph{centrally rigid} if  $c_\Gamma=|\textrm{Cent}(\gamma,\Gamma_{\textrm{Axis}(\gamma)})|$ is a constant on $P(\Gamma)$.
\end{defn}

\bigskip

 For example, if $\Gamma=PGL(2,\mathbb{F}_q[t])$, then   $|\textrm{Cent}(\gamma,\Gamma_{\textrm{Axis}(\gamma)})|=q-1$ for any $\gamma\in \Gamma$ due to Proposition~\ref{prop:pgl2conj}. The following theorem describes the relation between $Z_{(X,\mathcal{G})}(u)$ and $Z_{(X,w)}(u)$ in case of $\pi_1=\pi_1(X,\mathcal{G})$ is isomorphic to a centrally rigid subgroup of $\textrm{Aut}(\mathcal{T}_{q+1})$.
 
 \bigskip


\begin{thm}
Let $\mathbf{X}=(X,\mathcal{G})$ be a $(q+1)$-regular graph of groups arising as a quotient by a centrally rigid subgroup $\Gamma$ of $\textrm{Aut}(\mathcal{T}_{q+1})$ and $\mathbf{A}=(X,w)$ be the associated weighted graph. Then, we have
\[Z_\mathbf{X}(u)=Z_\mathbf{A}(u)^{c_\Gamma}.\]
\end{thm}
\begin{proof} By Equations~(\ref{eqn:Rm}) and (\ref{eqn:Nm}), we have
\begin{align*}
Z_\mathbf{A}(u)&=\,\exp\left(\sum_{m=1}^{\infty}N_m\frac{u^m}{m}\right)=\exp\left(\sum_{m=1}^{\infty}\sum_{\ell(\gamma)=m}\frac{\ell(\gamma_0)}{|\textrm{Cent}(\gamma,\Gamma_{\textrm{Axis}(\gamma)})|}\frac{u^m}{m}\right)\\
&=\,\exp\left(\frac{1}{c_\Gamma}\sum_{m=1}^{\infty}\sum_{\ell(\gamma)=m}\ell(\gamma_0)\frac{u^m}{m}\right)
=\exp\left(\sum_{m=1}^{\infty}\sum_{\ell(\gamma)=m}\ell(\gamma_0)\frac{u^m}{m}\right)^{\frac{1}{c_\Gamma}}\\
&=\,Z_\mathbf{X}(u)^{\frac{1}{c_\Gamma}}.
\end{align*}
This yields the statements of the theorem.
\end{proof}

It is shown in \cite{DK} that if $\Gamma$ is cuspidal, then the product $Z_X(u)$ converges to a rational function for small $|u|$. Therefore, if $(X,\mathcal{G})$ is geometrically finite and the fundamental group $\pi_1(X,\mathcal{G})$ is centrally rigid, then the Selberg zeta function $Z_{(X,\mathcal{G})}(u)$ converges to a rational function.

\bigskip

\subsection{Zeta function as a determinant}
Assume that $\Gamma$ is a cuspidal tree lattice acting on the tree $\mathcal{T}$ and let $\mathbf{A}=(X,w)$ be the associated weighted graph obtained by this action. Define the operator $T\colon Ed(X)\to Ed(X)$ by
\[T(e)=\sum_{e'}w(e,e')e'\]
where the sum runs over all edges $e'$ with $s(e')=t(e)$. Then, for any positive integer $n$, the operator $T^n$ is traceable, and the trace is given by 
\[\textrm{Tr}(T^n)=\sum_{C\colon \ell(C)=n}\ell(C_0)w(C)\]
where the sum runs over all cycles $C$ of length $n$ and $C_0$ is the underlying prime cycle of a given cycle $C$.

The determinant of the operator $I-uT$ is defined by the limit of the determinant of the restriction of $I-uT$ on a sequence of connected subgraphs $\{F\}$ with weights exhausting $X$, that is
\[\det(I-uT)=\lim_{F\rightarrow X} \det(I-uT|_F).\]
Theorem 4.3 in \cite{DK} shows that $\det(I-uT)$ exists for sufficiently small $|u|$ if $\Gamma$ is geometrically finite. Thus, we have the following Lemma.

\bigskip

\begin{lem}[Lemma 4.2 and Theorem 4.3 of \cite{DK}]
There is $\alpha>0$ such that the series 
\[-\sum_{n=1}^\infty \frac{u^n}{n}T^n\] converges weakly for $u\in\mathbb{C}$ with $|u|<\alpha$. If we denote by $\log(1-uT)$ the convergent operator, then this is traceable and we also have
\[Z_\mathbf{A}(u)=\frac{1}{\exp(\textrm{Tr}(\log(1-uT)))}=\frac{1}{\det(I-uT)}.\]

\end{lem}

\vspace{1em}

\begin{ex}
Let $\Gamma=PGL(2,\mathbb{F}_q[t])$. It is a lattice of $PGL(2,\mathbb{F}_q(\!(t^{-1})\!))$ whose Bruhat-Tits tree is the infinite $(q+1)$-regular tree. The quotient graph of groups $(X,\mathcal{G})$ under $\Gamma$-action is given by 
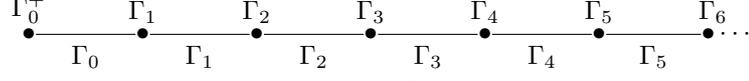
\begin{figure}[H]\label{fig:2}
\begin{center}
\begin{tikzpicture}[every loop/.style={}]
  \tikzstyle{every node}=[inner sep=0pt]
  \node (0) {$\bullet$} node [above=4pt] at (0,0) {$\Gamma_0^+$};
  \node (2) at (1.5,0) {$\bullet$} node [above=4pt] at (1.5,0) {$\Gamma_1$}; 
  \node (4) at (3,0) {$\bullet$}node [above=4pt] at (3,0) {$\Gamma_2$}; 
  \node (6) at (4.5,0) {$\bullet$}node [above=4pt] at (4.5,0) {$\Gamma_3$}; 
  \node (8) at (6,0) {$\bullet$}node [above=4pt] at (6,0) {$\Gamma_4$}; 
  \node (10) at (7.5,0) {$\bullet$}node [above=4pt] at (7.5,0) {$\Gamma_5$}; 
  \node (11) at (9.2,0) {$\bullet \cdots$}node [above=4pt] at (9,0) {$\Gamma_6$}; 

  \path[-] (0) edge node [below=4pt] {$\Gamma_0$} (2)
 (2) edge node [below=4pt] {$\Gamma_1$} (4)
(4) edge node [below=4pt] {$\Gamma_2$} (6)
 (6) edge node [below=4pt] {$\Gamma_3$} (8)
 (8) edge node [below=4pt] {$\Gamma_4$} (10)
 (10) edge node [below=4pt] {$\Gamma_5$} (11);
\end{tikzpicture}
\caption{The quotient graph of groups for $PGL(2,\mathbb{F}_q[t])$-action}
\end{center}
\end{figure}
Here, $\Gamma_0^+=PGL(2,\mathbb{F}_q)$, $\Gamma_0=\Gamma_0^+\cap\Gamma_1$ and 
\[\Gamma_n=\left\{\begin{pmatrix}a & b \\ 0 & d\end{pmatrix}\colon a,b\in\mathbb{F}_q^\times, b\in\mathbb{F}_q[t], \deg(b)\le n\right\}\textrm{ for }n\ge 1.\]
Let $\mathbf{A}=(X,w)$ be the associated weighted graph. Then, its Bass-Ihara zeta function is 
\[Z_\mathbf{A}(u)=\frac{1-qu^2}{1-q^2u^2}\]
while the Selberg zeta function for $\mathbf{X}=(X,\mathcal{G})$ is
\[Z_\mathbf{X}(u)=\left(\frac{1-qu^2}{1-q^2u^2}\right)^{q-1}.\]
\end{ex}

\vspace{1em}


\section{Examples with non-isomorphic isospectral graphs of groups}\label{sec:5}

In this section, we provide some computation results of zeta function for geometrically finite graphs of groups.
Especially, we give a proof of Proposition~\ref{prop:non-isom}.

For a positive integer $n$, let us denote by $C_n$ the cyclic group of order $n$. Consider
$G_0\subset G_1\subset G_2\subset \cdots$ be the infinite chain given by $G_0=C_k$ and $G_n\simeq  (C_q)^n$. Let $\mathbf{X}=(X,\mathcal{G})$ be the graph of groups described in Figure~\ref{fig:5}.

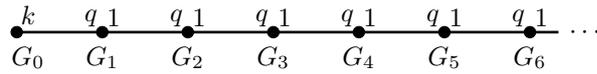
\begin{figure}[H]
\begin{center}
\begin{tikzpicture}[scale=0.75]
    
        \draw[line width=1pt,black] (-1,0.75)--(8.5,0.75);

        \draw[fill] (-1,0.75) circle (0.1);      
        \draw[fill] (0.5,0.75) circle (0.1);   
        \draw[fill] (2,0.75) circle (0.1);   
        \draw[fill] (3.5,0.75) circle (0.1);   
        \draw[fill] (5,0.75) circle (0.1);  
        \draw[fill] (6.5,0.75) circle (0.1);  
        \draw[fill] (8,0.75) circle (0.1);

        \node at (9,0.75) {$\cdots$};
        
        \node at (-0.8,0.3) {$G_0$}; 
        \node at (0.5,0.3) {$G_1$}; 
        \node at (2,0.3) {$G_2$}; 
        \node at (3.5,0.3) {$G_3$}; 
        \node at (5,0.3) {$G_4$}; 
        \node at (6.5,0.3) {$G_5$};
        \node at (8,0.3) {$G_6$};
                \node at (-0.8,1.05) {$k$}; 
           \node at (0.3,1) {$q$}; 
        \node at (0.7,1) {$1$}; 
        \node at (1.8,1) {$q$}; 
           \node at (2.2,1) {$1$}; 
        \node at (3.3,1) {$q$}; 
         \node at (3.7,1) {$1$}; 
        \node at (4.8,1) {$q$}; 
            \node at (5.2,1) {$1$}; 
        \node at (6.3,1) {$q$};
         \node at (6.7,1) {$1$};
        \node at (7.8,1) {$q$};
                \node at (8.2,1) {$1$};
\end{tikzpicture}
\caption{Graph of groups $\mathbf{X}=(X,\mathcal{G})$}\label{fig:5}

\end{center}

\end{figure}

Let $\Gamma=\pi_1(X,\mathcal{G})$ be the fundamental group of the graph of groups $(X,\mathcal{G})$. We note that $\Gamma$ is isomorphic to $C_k *C_q^{\infty}$.

\vspace{1em}

\begin{lem}\label{lem:centrigid}
The group $\Gamma$ is centrally rigid with $c_\Gamma=1$.
\end{lem}
\begin{proof}
Let $v_0,v_1,v_2,\ldots$ denote the vertices of the graph $X$. Every element $\delta$ of $\Gamma$ can be written as $\delta=a_1b_1a_2b_2\cdots a_nb_n$ for some $n\ge 1$ with $a_i\in C_k$, $b_i\in (C_q)^\infty$, allowing $a_1$ or $b_n$ may be the identity.

Consider the natural projection map $p\colon\mathcal{T}\to X$ from the universal covering tree $\mathcal{T}$ of $(X,\mathcal{G})$ to $X$. We note that for every bi-infinite geodesic in $\mathcal{T}$, there is at least one vertex $x$ such that $p(x)=v_0$.

Given a primitive hyperbolic element $\gamma$, choose a vertex $x$ in $\textrm{Axis}(\gamma)$ such that $p(x)=v_0$. If $\delta=a_1b_1\cdots a_nb_n\in\textrm{Cent}(\gamma,\Gamma_{\textrm{Axis}(\gamma)})$, then $b_j$ must be the identity for all $1\le i\le n$ since $\delta$ fixes the vertex $x$. Moreover, since $\delta$ fixes $\textrm{Axis}(\gamma)$ pointwise, it follows that $\delta$ itself is the identity.
\end{proof}
\bigskip

Hence, there is no essential difference between the Selberg zeta function for $\mathbf{X}$ and the Bass-Ihara zeta function of the weighted graph $\mathbf{A}$ associated to $\mathbf{X}$.

\bigskip

\begin{prop}
The Bass-Ihara zeta function for $\mathbf{A}=(X,w)$ associated to the above graph of groups is given by \[Z_\mathbf{A}(u)=\frac{1-qu^2}{1-(qk-k+1)u^2}.\]
\end{prop}

\begin{proof}Let $k$ be the vertex indexed by $G_k.$
Denote by $f_{2k-1}$ the edge from $v_{k-1}$ to $v_{k}$. Let $f_{2k}=\overline{f_{2k-1}}.$
The operator $T$ satisfies that
\begin{equation}\nonumber
Tf_k=\begin{cases}(k-1)f_1&\text{ if }i=2\\
                             (q-1)f_{k+1}+f_{k+2}&\text{ if } \text{ odd}\\
qf_{k-2} &\text{ if } k > 2 \text{ and even}.\end{cases}
\end{equation}
Let $a=-(q-1)u$ and $b=-u$. The determinant of $I-uT$ is the determinant of the following matrix
\begin{center}
\begin{tabular}{ c c| c c c c c c c c c c}
  $1$&$-(k-1)u$             &&&&&&&&&\\
              $a$ &   $1$       & &$a+b$& &       & & &&&\\
                \hline

              $b$&             &$1$&       & &      & & &&&\\
                 &             &$a$& $1$    & &$a+b$& & &&&\\
                  &             &$b$&       &$1$&      & & &&&\\
                  &             & &        &$a$&$1$    & &$a+b$&&\\
                  &             & &        &$b$&      &$1$& &&\\
                  &             & &        &&      &$a$&$1$&&&\\
                  &             & &        &&      &b&&&$\ddots$\\
\end{tabular}.
\end{center}
Following the proof of Lemma 4.3 in \cite{DK}, we have the determinant of $I-uT$ is equal to \[\det \begin{pmatrix}1 &-(k-1)u\\a&1\end{pmatrix}+\frac{a(a+b)}{1-(a+b)b}\det\begin{pmatrix}1&-(k-1)u\\b&0\end{pmatrix}.\]
Thus 
\begin{equation*}
\begin{split}
    \det(I-uT)&=1+a(k-1)u+\frac{ab(a+b)(k-1)u}{1-b(a+b)}\\
    &=\frac{1-b(a+b)+a(k-1)u}{1-b(a+b)}\\
    &=\frac{1-(qk-k+1)u^2}{1-qu^2}.
    \end{split}
\end{equation*}
This yields the proposition.
\end{proof}

Let $\alpha_i$ be the integers with $k:=\alpha_1+\cdots+\alpha_n\leq q+1.$ Let $\mathbf{A}^{(n)}(\alpha_1,\ldots,\alpha_n)$ be a weighted graph arising from a geometrically finite lattice which consists of $n$-rays connected at a base point $v_0$ with indices $\alpha_1,\ldots,\alpha_n$. See Figure~\ref{fig:6} for $\mathbf{A}^{(3)}(\alpha_1,\alpha_2,\alpha_3)$. 

\bigskip

\begin{prop}\label{prop:ex2}
If $\alpha_1+\alpha_2+\cdots+\alpha_n=k$, then the Bass-Ihara zeta function of $\mathbf{A}^{(n)}(\alpha_1,\ldots,\alpha_n)$ is given by
\[Z_\mathbf{A}(u)=\frac{(1-qu^2)^n}{(1-u)^{n-1}(1+u)^{n-1}(1-(kq-k+1)u^2)}.\]
\end{prop}
\begin{proof}

Let $v_{mn+r}$ be the point on $r$-th ray with $d(v_0,v_{mn+r})=m+1.$ The edges $e_{2r-1}$ are the edges from $v_0$ to $v_r$ and $e_{2(m+1)n+2r-1}$ from $v_{mn+r}$ to $v_{(m+1)n+r}$. Let $\overline{e_{2i}}=e_{2i+1}.$

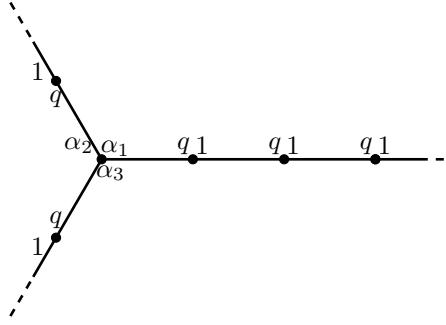
\begin{figure}[h]
\begin{center}
\begin{tikzpicture}[scale=0.6]
    
        \draw[line width=1pt,black] (-1,0.0)--(6,0.0);
        \draw[line width=1pt,black,dashed] (6.5,0.0)--(6,0.0);
     \draw[line width=1pt,black] (-1,0.0)--(-2.5,{sqrt(6.75)});
          \draw[line width=1pt,black,dashed] (-3,{sqrt(12)})--(-2.5,{sqrt(6.75)});
       \draw[line width=1pt,black,dashed] (-3,-{sqrt(12)})--(-2.5,-{sqrt(6.75)});
         \draw[line width=1pt,black] (-1,0.0)--(-2.5,-{sqrt(6.75)});   
        \draw[fill] (-2,+{sqrt(3)}) circle (0.1);      
        \draw[fill] (-2,-{sqrt(3)}) circle (0.1);      
        \draw[fill] (-1,0.0) circle (0.1);      
        \draw[fill] (1,0.0) circle (0.1);   
        \draw[fill] (3,0.0) circle (0.1);   
        \draw[fill] (5,0.0) circle (0.1);   

        \draw[fill] (5,0.0) circle (0.1);  

        \node at (-0.7,0.25) {$\alpha_1$}; 
        \node at (-1.5,0.3) {$\alpha_2$}; 
        \node at (-0.8,-0.3) {$\alpha_3$}; 
        \node at (0.8,0.3) {$q$}; 
        \node at (-2.4,+{sqrt(3)}+0.2) {$1$}; 
        \node at (-2.0,+{sqrt(3)}-0.4) {$q$}; 
        \node at (-2,-{sqrt(3)}+0.4) {$q$}; 
        \node at (-2.4,-{sqrt(3)}-0.2) {$1$}; 
        \node at (1.2,0.3) {$1$}; 
     \node at (2.8,0.3) {$q$}; 
        \node at (3.2,0.3) {$1$}; 
         \node at (4.8,0.3) {$q$}; 
        \node at (5.2,0.3) {$1$}; 
\end{tikzpicture}

\caption{Weighted graph $\mathbf{A}^{(3)}(\alpha_1,\alpha_2,\alpha_3)$}\label{fig:6}
\end{center}
\end{figure}

The operator $T$ satisfies that
\begin{equation}\nonumber
Te_i=\begin{cases}\displaystyle\sum_{j=1}^n\alpha_je_{2j-1}-e_{i-1}&\text{ if }i=2r\text{ and } 1\leq r\leq n\\
                             (q-1)e_{i+1}+e_{i+2n}&\text{ if } \text{ odd}\\
qe_{i-2n} &\text{ if } i > 2n \text{ and even}.\end{cases}
\end{equation}
 Let $a=-(q-1)u$, $b=-u$ and $a_n=a\sum_{k=0}^{n-1}b^k(a+b)^k$. Let $\delta_{st}$ be the Kronecker delta function.
Let $A_{n,k}$, $B_n$, $C_n$, $D_{n,k}$ and $E_{n,k}$ be the $2n\times 2n$ matrices defined by
\begin{equation*}
\begin{split}
&(A_{n,k})_{ij}=\begin{cases}1&\text{ if } i=j\\
   a_k&\text{ if } (i,j)=(2s,2s-1)\\
   -(\alpha_s-\delta_{st})u&\text{ if } (i,j)=(2s-1,2t)\\
   0&\text{ otherwise}
\end{cases}\\
&(B_{n})_{ij}=\begin{cases}
    b &\text{ if } (i,j)=(2s-1,2s-1)\\
    0 &\text{ otherwise},
\end{cases}\quad
(C_{n})_{ij}=\begin{cases}
    a+b &\text{ if } (i,j)=(2s,2s)\\
    0 &\text{ otherwise},\\
\end{cases}\\
  &(D_{n,k})_{ij}=\begin{cases}  1 &\text{ if } i=j\\
  a_k&\text{ if } (i,j)=(2s,2s-1)\\
    0 &\text{otherwise}
\end{cases}\\
&\text{and}\\
&(E_{n,k})_{ij}=\begin{cases}  a+b &\text{ if } (i,j)=(2s,2s)\\
  -a_k(a+b)&\text{ if } (i,j)=(2s,2s-1)\\
    0 &\text{otherwise}.
    \end{cases}
\end{split}
\end{equation*}
By Theorem 4.3 in \cite{DK}, the determinant $I-uT$ is the limit of the determinant of the $N\times N$ block matrix
\begin{equation*}
    (M_N)_{ij}=\begin{cases}
    A_{n,1}& \text{ if } i=j=1\\
        B_n&\text{ if } i-j=1\\
    C_{n}&\text{ if } i-j=-1\\
    D_{n,1}&\text{ if } i=j\neq 1\\  
    0&\text{ otherwise}.
    \end{cases}
\end{equation*}
We employ the column operations to calculate the determinant of $M_N$. 
Denote by $[M_N]_j$ the $j$-th column of $M_N$ and by $(M_N)_{m,n}$ the matrix in $(m,n)$-th entry of the block matrix $M_N$. We execute Step 1 and Step 2 sequentially for $i$ ranging from $1$ to $N-1$.
\begin{itemize}
\item  \textit{Step 1} To change $(M_N)_{N-i+1,N-i+1}$ to $I$ and to change $(M_N)_{N-i,N-i+1}$ to $E_{n,i}$, we perform that for any positive integer $s\leq n$,
\begin{equation*}\label{eq:4.1}
\begin{split} 
[M_N]_{k(N-i)+2s-1}\longrightarrow[M_N]_{k(N-i)+2s-1}-a_i[M_N]_{k(N-i)+2s}.
\end{split}
\end{equation*}
\item \textit{Step 2} To change $(M_{N})_{N-i+1,N-i}$ to the zero matrix and to change $(M_{N})_{N-i,N-i}$ to $D_{n,i+1}$ or $A_{n,N}$, we perform that for any positive integer $s\leq n$,
\begin{equation*}\label{eq:4.2}
[M_N]_{k(N-i-1)+2s-1}\longrightarrow[M_N]_{k(N-i-1)+2s-1}-b[M_N]_{k(N-i)+2s-1}.
\end{equation*}
\end{itemize}

Let $A_n=\lim_{N\rightarrow \infty} A_{n,N}=$ and $\alpha:=\lim_{k\rightarrow \infty}a_k$. Then
\begin{equation*}
    (A_n)_{ij}
=\begin{cases}
\alpha&\text{ if } (i,j)=(2s,2s-1)\\
    (A_{n,k})_{ij} &\text {otherwise}.
\end{cases}\end{equation*}
    As a result of the operation, we have 
\begin{equation*}
\begin{split}
    \det(I-uT)&=\lim_{N\rightarrow\infty}\det(M_N)= \det(A_{n}).
    \end{split}
\end{equation*}
The remaining part of the proof is to show that 
$$\det(A_{n})=\frac{(1-u)^{n-1}(1+u)^{n-1}(1-(kq-k+1)u^2)}{(1-qu^2)^n}.$$
Let $[A]_i$ and $[A]^j$ be the $i$-th row and $j$-th column of a matrix $A$. To prove the above equality, we perform the following operations.
\begin{itemize}
    \item \underline{Step 1} Let $F_1$ be the matrix obtained by the following operation: for any $i$, 
\begin{equation*}
[A_n]_{2i-1}\longrightarrow [A_n]_{2i-1}-u[A_n]_{2i}.
\end{equation*}
\item \underline{Step 2} Let $F_2$ be the matrix obtained by the following operation: for any $j$
\begin{equation*}
    [F_1]^{2j}\longrightarrow[F_1]^{2j}-[F_1]^{2j+2}.
\end{equation*}
\item \underline{Step 3} Let $F_3$ be the matrix obtained by the following operation: for any $i$
\begin{equation*}
    [F_2]_{2i}\longrightarrow [F_2]_{2i}-\frac{\alpha}{1-\alpha u}[F_2]_{2i-1}.
\end{equation*}
\item \underline{Step 4} Let $F_4$ be the matrix obtained by the following operation: for any $i$
\begin{equation*}
    [F_3]_{2i+2}\longrightarrow [F_3]_{2i+2}+[F_3]_{2i}.
\end{equation*}
\end{itemize}
The matrix $F_4$ is an upper triangular matrix with diagonal entries satisfying 
\begin{equation*}
    (F_4)_{ii}=\begin{cases}
        1-\alpha u &\text{ if } i=2s-1\\
        1 &\text{ if } i=2s\text{ and } s\neq n\\
        1+\dfrac{k\alpha}{1-\alpha u} &\text{ if } s=2n.
    \end{cases}
\end{equation*}
Since $\alpha =\frac{a}{1-(a+b)b}=\frac{-(q-1)u}{1-qu^2},$ it follows that
\begin{equation*}
\begin{split}
\det(A_{n})&=(1-\alpha u)^n\biggl( 1+\frac{k\alpha}{1-\alpha u}\biggr)\\
&=\frac{(1-u)^{n-1}(1+u)^{n-1}(1-(kq-k+1)u^2)}{(1-qu^2)^n}.
\end{split}
\end{equation*}
This completes the proof.
\end{proof}

\bigskip

\begin{coro}

For every $q\ge 3$, there are infinitely many pairs of $(q+1)$-regular graphs of groups whose induced weighted graphs are non-isomorphic but the Bass-Ihara zeta functions are identical.
\end{coro}
\begin{proof}
By Proposition~\ref{prop:ex2}, the zeta functions of $\mathbf{A}^{(n)}(\alpha_1,\ldots,\alpha_n)$ depends only on the number of cusps $n$ and the sum of $\alpha_i$'s. Hence, for example,
\[\mathbf{A}^{(2)}(q,1)\qquad\textrm{ and }\qquad \mathbf{A}^{(2)}(q-2,2)\]
have the same zeta functions for every $q\ge 3$.
\end{proof}

\bigskip


\section{Sequence $(X_N,\mathcal{G}_N)$ of graphs of groups towards small pole-free region}\label{sec:6}

In this section, we prove Theorem~\ref{thm:pole-free}. Recall that a pole-free region of $Z_\mathbf{X}(u)$ is the region $\{u\in\mathbb{C}\colon |u|<R_\mathbf{X}+\epsilon\}$ where any pole of $Z_\mathbf{X}(u)$ in the region $R_\mathbf{X}\le |u|\le R_\mathbf{X}+\epsilon$ must lie on the circle $|u|=R_\mathbf{X}$.

Let us denote by $C_n$ the cyclic group of order $n$. We consider the filtration
\[H_N\supset H_{N-1}\supset\cdots\supset H_1\supset H_0\]
with $H_k\simeq C_{q-1}\times C_q^k$.
Similarly, let 
$G_1\subset G_2\subset \cdots$ be the infinite chain given by $G_k\simeq (C_q)^k$. Let $(X_N,\mathcal{G}_N)$ be the graph of groups defined as Figure~\ref{fig:1}.

\begin{center}
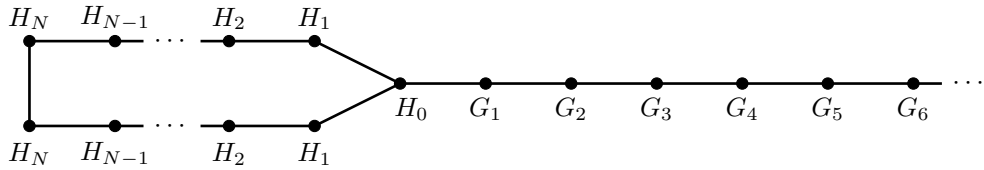
\begin{figure}[h]
\begin{tikzpicture}[scale=0.75]
        \draw[line width=1pt,black] (-1,0.75)--(-2.5,0) -- (-4.5,0);
        \draw[line width=1pt,black] (-4.5,1.5)--(-4,1.5)--(-2.5,1.5)--(-1,0.75)--(8.5,0.75);
        \draw[line width=1pt,black]
        (-5.5,1.5)--(-7.5,1.5)--(-7.5,0)--(-5.5,0);
            
        \draw[fill] (-4,0) circle (0.1);  
        \draw[fill] (-2.5,0) circle (0.1);  
        \draw[fill] (-2.5,1.5) circle (0.1);  
        \draw[fill] (-4,1.5) circle (0.1);
        \draw[fill] (-1,0.75) circle (0.1);      
        \draw[fill] (0.5,0.75) circle (0.1);   
        \draw[fill] (2,0.75) circle (0.1);   
        \draw[fill] (3.5,0.75) circle (0.1);   
        \draw[fill] (5,0.75) circle (0.1);  
        \draw[fill] (6.5,0.75) circle (0.1);  
        \draw[fill] (8,0.75) circle (0.1); 
        \draw[fill] (-7.5,0) circle (0.1);
        \draw[fill] (-7.5,1.5) circle (0.1);
        \draw[fill] (-6,0) circle (0.1);
        \draw[fill] (-6,1.5) circle (0.1);
        
        \node at (-5,0) {$\cdots$};
        \node at (-5,1.5) {$\cdots$};
        \node at (9,0.75) {$\cdots$};
        
        \node at (-0.8,0.3) {$H_0$}; 
        \node at (0.5,0.3) {$G_1$}; 
        \node at (2,0.3) {$G_2$}; 
        \node at (3.5,0.3) {$G_3$}; 
        \node at (5,0.3) {$G_4$}; 
        \node at (6.5,0.3) {$G_5$};
        \node at (8,0.3) {$G_6$};
       
        \node at (-4,1.9) {$H_2$};
        \node at (-2.5,1.9) {$H_1$}; 
        \node at (-7.5,1.9)  {$H_N$};
        \node at (-7.5,-0.5) {$H_N$};
        \node at (-4,-0.5) {$H_2$};
        \node at (-2.5,-0.5) {$H_1$}; 
        \node at (-6,1.9) {$H_{N-1}$};
        \node at (-6,-0.5) {$H_{N-1}$};
\end{tikzpicture}
\caption{Graph of groups $(X_N,\mathcal{G}_N)$}\label{fig:1}
\end{figure}
\end{center}

Similar arguments with Lemma~\ref{lem:centrigid} yields that $\Gamma=\pi_1(X_N,\mathcal{G}_N)$ is a centrally rigid subgroup with $c_\Gamma=1$. The weighted graph associated to $(X_N,\mathcal{G}_N)$ is described in Figure \ref{fig:2}. Now, let us describe the operator $T\colon Ed(X_N)\to Ed(X_N)$.

To facilitate the description, we introduce the following notation for the vertices in $(X_N,\mathcal{G}_N)$ as shown in Figure \ref{fig:3}: $a_k$ and $b_k$ are $k$-th vertices in upper and lower layer, respectively, and $c_k$ is the $k$-th vertex in the central layer. We denote by $e_1$ the oriented edge from $a_1$ to $c_1$, and $e_{2N+1}$ the oriented edge from $b_N$ to $a_N$. Then, we define $e_2=\overline{e_1}$ and $e_{2N+2}=\overline{e_{2N+1}}$.

For $1\leq k \leq N-1$, we let $e_{2k+1}$ be the oriented edge from $a_{k+1}$ to $a_k$, and $e_{2k+2}=\overline{e_{2k+1}}$. For $N+1\leq k \leq 2N-1$, we let $e_{2k+1}$ be the oriented edge from $b_{2N-k}$ to $b_{2N-k+1}$. We denote by $e_{4N+1}$ the oriented edge from $c_1$ to $b_1$, and $e_{4N+1}=\overline{e_{4N+2}}$.

Finally, we denote by $f_{2k-1}$ the edge from $c_k$ to $c_{k+1}$, and $f_{2k-1}=\overline{f_{2k}}$.


The operator $T$ satisfies that
\begin{equation}\nonumber
\begin{split}
&Te_1=e_{4N+1}+(q-1)f_1,\quad Te_{4N+2}=e_2+(q-1)f_1\\
&Te_{2k+1}=\begin{cases}qe_{2k-1} &\text{ if } 1\leq k \leq N\\
                                         e_{2k-1}+(q-1)e_{2k+2} &\text{ if }\quad N+1\leq k\leq 2N
\end{cases}\\
&Te_{2k+2}=\begin{cases}(q-1)e_{2k+1}+e_{2k+4} &\text{ if } 0\leq k\leq N-1\\
                                          qe_{2k+4}      & \text{ if } N\leq k\leq 2N-1
\end{cases}\\
&Tf_k=\begin{cases}e_2+e_{4N+1}+(q-2)f_1&\text{ if }i=2\\
                             (q-1)f_{k+1}+f_{k+2}&\text{ if } \text{ odd}\\
qf_{k-2} &\text{ if } k > 2 \text{ and even}.\end{cases}
\end{split}
\end{equation}

\begin{center}
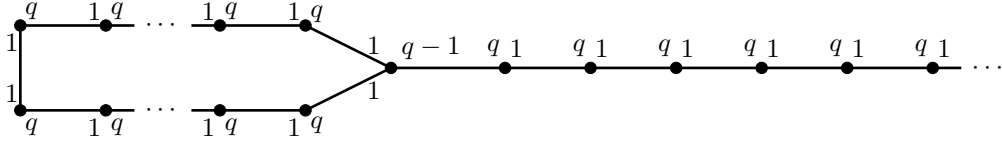
\begin{figure}[h]
\begin{tikzpicture}[scale=0.75]
        \draw[line width=1pt,black] (-1,0.75)--(-2.5,0) -- (-4.5,0);
        \draw[line width=1pt,black] (-4.5,1.5)--(-4,1.5)--(-2.5,1.5)--(-1,0.75)--(9,0.75);
        \draw[line width=1pt,black]
        (-5.5,1.5)--(-7.5,1.5)--(-7.5,0)--(-5.5,0);
            
        \draw[fill] (-4,0) circle (0.1);  
        \draw[fill] (-2.5,0) circle (0.1);  
        \draw[fill] (-2.5,1.5) circle (0.1);  
        \draw[fill] (-4,1.5) circle (0.1);
        \draw[fill] (-1,0.75) circle (0.1);      
        \draw[fill] (1,0.75) circle (0.1);   
        \draw[fill] (2.5,0.75) circle (0.1);   
        \draw[fill] (4,0.75) circle (0.1);   
        \draw[fill] (5.5,0.75) circle (0.1);  
        \draw[fill] (7,0.75) circle (0.1);  
        \draw[fill] (8.5,0.75) circle (0.1); 
        \draw[fill] (-7.5,0) circle (0.1);
        \draw[fill] (-7.5,1.5) circle (0.1);
        \draw[fill] (-6,0) circle (0.1);
        \draw[fill] (-6,1.5) circle (0.1);
        
        \node at (-5,0) {$\cdots$};
        \node at (-5,1.5) {$\cdots$};
        \node at (9.5,0.75) {$\cdots$};
       
        \node at (-7.3,-0.3) {$q$};\node at (-6.2,-0.3) {$1$};\node at (-7.65,0.3) {$1$};\node at(-7.65,1.2) {$1$};\node at (-6.2,1.75) {$1$};\node at(-7.3,1.75) {$q$};\node at (-5.8,-0.3) {$q$};\node at (-5.8,1.75) {$q$};\node at(-4.2,1.75) {$1$};\node at(-4.2,-0.3) {$1$};\node at(-3.8,-0.3) {$q$};\node at(-3.8,1.75) {$q$};\node at(-2.7,1.75) {$1$};\node at(-2.7,-0.3) {$1$};\node at(-2.3,1.7) {$q$};\node at(-2.3,-0.25) {$q$};\node at(-1.3,0.35) {$1$};\node at (-1.3,1.15) {$1$}; \node at (-0.3,1.1) {$q-1$};\node at (0.8,1.1) {$q$};\node at (1.2,1.1) {$1$};\node at (2.3,1.1) {$q$};\node at (2.7,1.1) {$1$};\node at (3.8,1.1) {$q$};\node at (4.2,1.1) {$1$};\node at (5.3,1.11) {$q$};\node at (5.7,1.1) {$1$};\node at (6.8,1.1) {$q$};\node at (7.2,1.1) {$1$};\node at (8.3,1.1) {$q$};\node at (8.7,1.1) {$1$};
        \node at (-4.2,1.9) {};
        \node at (-2.2,1.9) {};         
     
\end{tikzpicture}
\caption{Weighted graph associated to $(X_N,\mathcal{G}_N)$}\label{fig:2}
\end{figure}
\end{center}

\begin{center}
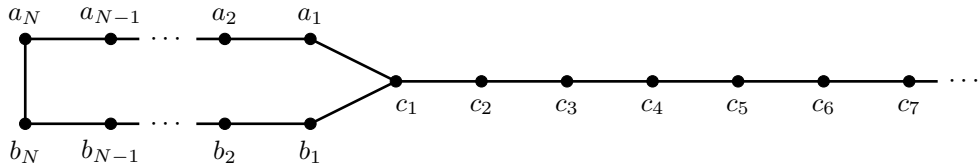
\begin{figure}[h]
\begin{tikzpicture}[scale=0.75]
        \draw[line width=1pt,black] (-1,0.75)--(-2.5,0) -- (-4.5,0);
        \draw[line width=1pt,black] (-4.5,1.5)--(-4,1.5)--(-2.5,1.5)--(-1,0.75)--(8.5,0.75);
        \draw[line width=1pt,black]
        (-5.5,1.5)--(-7.5,1.5)--(-7.5,0)--(-5.5,0);
            
        \draw[fill] (-4,0) circle (0.1);  
        \draw[fill] (-2.5,0) circle (0.1);  
        \draw[fill] (-2.5,1.5) circle (0.1);  
        \draw[fill] (-4,1.5) circle (0.1);
        \draw[fill] (-1,0.75) circle (0.1);      
        \draw[fill] (0.5,0.75) circle (0.1);   
        \draw[fill] (2,0.75) circle (0.1);   
        \draw[fill] (3.5,0.75) circle (0.1);   
        \draw[fill] (5,0.75) circle (0.1);  
        \draw[fill] (6.5,0.75) circle (0.1);  
        \draw[fill] (8,0.75) circle (0.1); 
        \draw[fill] (-7.5,0) circle (0.1);
        \draw[fill] (-7.5,1.5) circle (0.1);
        \draw[fill] (-6,0) circle (0.1);
        \draw[fill] (-6,1.5) circle (0.1);
        
        \node at (-5,0) {$\cdots$};
        \node at (-5,1.5) {$\cdots$};
        \node at (9,0.75) {$\cdots$};
        
        \node at (-0.8,0.3) {$c_1$}; 
        \node at (0.5,0.3) {$c_2$}; 
        \node at (2,0.3) {$c_3$}; 
        \node at (3.5,0.3) {$c_4$}; 
        \node at (5,0.3) {$c_5$}; 
        \node at (6.5,0.3) {$c_6$};
        \node at (8,0.3) {$c_7$};
       
        \node at (-4,1.9) {$a_2$};
        \node at (-2.5,1.9) {$a_1$}; 
        \node at (-7.5,1.9)  {$a_{N}$};
        \node at (-7.5,-0.5) {$b_N$};
        \node at (-4,-0.5) {$b_2$};
        \node at (-2.5,-0.5) {$b_1$}; 
        \node at (-6,1.9) {$a_{N-1}$};
        \node at (-6,-0.5) {$b_{N-1}$};
\end{tikzpicture}
\caption{Labeling vertices}\label{fig:3}
\end{figure}
\end{center}

Let $a=-(q-1)u$ and $b=-u$. Denote by $A_N$ the matrix representation of $I-uT$ on the subgraph consisting of edges $e_1,e_2,\dots,e_{4N+2}$, $f_1$ and $f_2$. The number $i$ corresponds to $e_{i}$ when $1 \leq i \leq 4N+2$, and $4N+3$ and $4N+4$ correspond to $f_1$ and $f_2$, respectively. The above equation for $T$ shows that if the $(i,j)$-entry of $A_n$ is nonzero and $i\neq j$, then $(i,j)$ is in the one of the following sets
\begin{equation}\nonumber
\begin{split}
I_1&:=\{(2k-1,2k):1\leq k \leq N\}\cup\{(2k,2k-1):N+2\leq k \leq 2N+1\}\\
&\quad\quad\cup\{(4N+3,1),(4N+3,4N+2),(4N+4,4N+3)\}\\
I_2&:=\{(2k+2,2k): 1\leq k \leq N\}\cup\{(2k+1,2k+3):N\leq k \leq 2N-1\}\\
&\quad\quad\{(2,4N+2),(2,4N+4),(4N+1,1),(4N+1,4N+4)\}\\
I_3&:=\{(2k+1,2k+3):0\leq k\leq N-1\}\cup\{(2k+2,2k):N+1\leq k\leq 2N\}\\
I_4&:=\{(4N+3,4N+4)\}.
\end{split}
\end{equation}
Then the entries $(A_N)_{ij}$ of $A_N$ satisfies that
\begin{equation}\nonumber
(A_N)_{ij}=\begin{cases}1 &\text{ if } i=j\\
                                a &\text{ if } (i,j)\in I_1\\
                                b &\text{ if } (i,j)\in I_2\\
                                a+b &\text{ if } (i,j)\in I_3\\
                                a-b&\text{ if }(i,j)= (4k+3,4k+4)\\
                                0 &\text{ otherwise}.
\end{cases}
\end{equation}
Then the determinant of $I-uT$ is the determinant of the following matrix

\bigskip

\begin{center}
\begin{tabular}{ c|c c c c c c c c c c c}
  $A_N$   &  & $\beta$&&&&&&&\\
  \hline
  $\alpha$&$1$&             &&&&&&&&\\
               & $a$ &   $1$       &0&$a+b$& &       & & &&\\
               & $b$&             &$1$&       & &      & & &&\\
               &   &             &$a$& $1$    & &$a+b$& & &&\\
               &   &             &$b$&       &$1$&      & & &&\\
               &   &             & &        &$a$&$1$    & &$a+b$&&\\
               &   &             & &        &$b$&      &$1$& &\\
               &   &             & &        &&      &$a$&$1$&&\\
               &   &             & &        &&      &b&&&$\ddots$\\
\end{tabular}.
\end{center}

\vspace{1em}

\noindent where $\alpha$ is the row vector with $\alpha_i=0$ for $i\neq 4N+3$ and $\alpha_{4N+3}=b$ and $\beta$ is the column vector with $\beta_j=0$ for $j\ne 4N+4$ and $\beta_{4N+4}=a+b$. Following the proof of Lemma 4.2 in \cite{DK}, we have the determinant of $I-uT$, which is $\det A_N$ plus $\frac{-ab(a+b)}{1-(a+b)b}\det A_N'$ where $A_N'$ is the submatrix of $A_N$ which is removed $(4N+4)$-th row and $(4N+3)$-th column. Let us compute $\det (A_N')$ and $\det(A_N)$. We use the row operations on $A_N'$ and $A_N$. Both of matrices, we use the same process as follows. 
\begin{itemize}
    \item Step 1. Denote by $[A_N']_i$ the $i$-th row of $A_N'$. Let   $a_k=a\sum_{i=0}^{k-1}b^i(a+b)^i$. For $2\leq i \leq N+1$, we apply the first operation and for $N+2\leq i \leq 2N$, we apply the second operation.
    \begin{equation*}
        \begin{split}
        &[A_N']_{2i}\rightarrow B_i:=[A_N']_{2i}-bB_{i-1} \\
        &[A_N']_{2i}\rightarrow B_i:=[A_N']_{2i}-a_{i-N-1}[A_N']_{2i-1}-(a+b)B_{i-1}.
        \end{split}
    \end{equation*}We change $[A_N']_{4N+2}$ to $[A_N']_{4N+2}-(a+b)B_{2N}.$
    \item Step 2.  Let $C_1=[A_{4N-1}]$. For $2\leq i \leq N$, we apply the first operation and for $N+1\leq i \leq 2N,$ we apply the second operation.
    \begin{equation*}
        \begin{split}
            &[A_{N}']_{4N-2i+1}\rightarrow C_i:=[A_{N}']_{4N-2i+1}-bC_{i-1}\\
            &[A_{N}']_{4N-2i+1}\rightarrow C_i:=[A_{N}']_{4N-2i+1}-aB_{4N-2i}-(a+b)C_{i-1}.
        \end{split}
    \end{equation*}
\end{itemize}   
After changing $[A_N']_{4N+1}$ to $[A_N']_{4N+1}-bC_{2N}$ and $[A_n']_{4N+3}$ to $[A_N']_{4N+3}-aC_{2N}$, 
we conclude that $\det(A_N')$ is equal to
\begin{equation}\nonumber
\begin{split}
\det\begin{pmatrix}
1+b^{N+1}(a+b)^N&ab^2\sum_{k=0}^{N-1} b^k(a+b)^k&b\\
a\sum_{k=0}^{N-1}b^k(a+b)^k&1+b^{N+1}(a+b)^N&-1\\
ab^N(a+b)^N&a+a^2b\sum_{k=0}^{N-1}b^k(a+b)^k&-b
 \end{pmatrix}
 \end{split}
 \end{equation}
which reduces to
 \begin{equation}\nonumber
 \begin{split}
 \biggl(1+b^{N+1}(a+b)^N+ab\sum_{k=0}^{N-1}b^k(a+b)^k\biggr)\biggl(a\sum_{k=0}^{N}b^k(a+b)^{k}-b-b^{N+1}(a+b)^{N+1}\biggr).
 \end{split}
 \end{equation}
 Let us perform the same row operations for $A_n$. This process shows that 
 \begin{equation}\begin{split}
 \det(A_N)&=\det\begin{pmatrix}\nonumber
 1+b^{N+1}(a+b)^N&ab^2\sum_{k=0}^{N-1}b^k(a+b)^k&0&b\\
 a\sum_{k=0}^{N-1}b^k(a+b)^k&1+b^{N+1}(a+b)^N&0&-1\\
 ab^N(a+b)^N&a+a^2b
 \sum_{k=0}^{N-1}b^k(a+b)^k&1&-b\\
 0&0&a&1
 \end{pmatrix}\\
 &=(1+b^{N+1}(a+b)^N)^2-a^2b^2\biggl(\sum_{k=0}^{N-1}b^k(a+b)^k\biggr)^2-a\det(A_N')
  \end{split}
 \end{equation}
 The determinant of $I-uT$ (which is the reciprocal of $Z_{\mathbf{X}_N}(u)$) is
 \begin{equation}\nonumber
 \begin{split}
 &Z_{\mathbf{X}_N}(u)^{-1}=\det (A_N)+\frac{q(q-1)u^3}{1-qu^2}\det(A_N')
\end{split}
\end{equation}
which factors to
\[\frac{(1-u^2)(1-qu)\left\{(1+u)\displaystyle\sum_{k=0}^{N-1}q^ku^{2k}+q^Nu^{2N}\right\}\biggl(1+(q-1)\displaystyle\sum_{k=0}^{N}q^ku^{2k+1
         }-q^nu^{2N+1}\biggr)}{1-qu^2}.\]

The trivial pole $u=1/q$ of $Z_{\mathbf{X}_N}(u)$ appears from the factor $1-qu$ and the second pole (which is negative real) appears at the last factor
\[1+\sum_{k=0}^{N-1}(q-1)q^ku^{2k}+q^N(q-2)u^{2N+1}.\] It converges to $\frac{1}{q}$ as $N\to \infty$ and hence the pole-free regions of $Z_{\mathbf{X}_N}$ converges to $|u|\le \frac{1}{q}$.

\bigskip
In particular, it does not satisfy Ramanujan property of \cite{Mo}. Hence, for each $N\ge 1$, $\pi_1(X_N,\mathcal{G}_N)$ is not a subgroup of $PGL_2(\mathbb{F}_q(\!(t^{-1})\!))$. Additionally, we have
\bigskip

\begin{coro} For any given $\epsilon>0$, there is $N\ge 1$ such that
\[|N_m(\mathbf{X}_N)-q^m|\notin O((q-\epsilon)^m)\] as $m\to\infty$.
\end{coro}
\begin{proof}
We may write
\[Z_{\mathbf{X}_N}(u)=\frac{(1-\sqrt{q}u)(1+\sqrt{q}u)}{(1+u)(1-u)(1-qu)\prod_{i=1}^{2N+1}(1-\alpha_iu)\prod_{j=1}^{2N+1}(1-\beta_ju)}\] with $|\alpha_1|\le\cdots\le |\alpha_{2N+1}|$ and $|\beta_1|\le \cdots\le |\beta_{2N+1}|$.
Since $c_\Gamma=1$, it follows from Equation~(\ref{eqn:Rm}) that
\[u\frac{Z_{\mathbf{X}_N}'(u)}{Z_{\mathbf{X}_N}(u)}=\sum_{m=1}^{\infty}N_m({\mathbf{X}_N})u^m.\]
Then, \begin{align*}
u\frac{d}{du}(\log Z_{\mathbf{X}_N}(u))=\frac{-2qu}{1-qu^2}+\frac{2u}{1-u^2}+\frac{q}{1-qu}+\sum_{i=1}^{2N+1}\frac{\alpha_i}{1-\alpha_iu}+\sum_{j=1}^{2N+1}\frac{\beta_j}{1-\beta_ju}
\end{align*}which yields $N_m(\mathbf{X}_{N})=q^m+\sum_{i=1}^{2N+1}(\alpha_i^m+\beta_i^m)-q^{\frac{m}{2}}-(-1)^mq^{\frac{m}{2}}+1+(-1)^m$. Since $|\alpha_i|,|\beta_i|<q$ and $|\beta_{2N+1}-q|\to 0$ as $N\to\infty$, we get the statement.
\end{proof}




\end{document}